\documentclass{article}

\usepackage{amsmath}
\usepackage{amssymb}
\usepackage{amsthm}
\usepackage{enumerate}

\numberwithin{equation}{section}  

\newtheorem{lem}{Lemma}[section]
\newtheorem{thm}[lem]{Theorem}
\newtheorem{dfn}[lem]{Definition}
\newtheorem{rmk}[lem]{Remark}
\newtheorem{prp}[lem]{Proposition}
\newtheorem{cor}[lem]{Corollary}
\newtheorem{exm}[lem]{Example}

\newcommand{\ii}{\sqrt{-1}}   
\newcommand{\ba}{\bar{\alpha}}   
\newcommand{\bb}{\bar{\beta}}
\newcommand{\bg}{\bar{\gamma}}

\begin{document}

\title{Liouville theorem for Pseudoharmonic maps from Sasakian manifolds\footnote{Supported by NSFC grant No. 11271071} }
\author{Yibin Ren\footnote{School of Mathematical Science, Fudan University, Shanghai 200433, P.R.China. E-mail: allenrybqqm@hotmail.com} \and Guilin Yang \and Tian Chong}
\date{}
\maketitle

\begin{abstract}
In this paper, we derive a sub-gradient estimate for pseudoharmonic maps from noncompact complete Sasakian manifolds which satisfy CR sub-Laplace comparison property, to simply-connected Riemannian manifolds with nonpositive sectional curvature. As its application, we obtain some Liouville theorems for pseudoharmonic maps. In the Appendix, we modify the method and apply it to harmonic maps from noncompact complete Sasakian manifolds.
\end{abstract}

\section{Introduction}

In \cite{y}, S. T. Yau derived a well-known gradient estimate for harmonic functions on complete noncompact Riemnnian manifolds. By this estimate, he got a Liouville theorem for positive harmonic functions on Riemannian manifolds with nonnegative Ricci curvature. In \cite{c}, S. Y. Cheng generlized the method in \cite{y} to harmonic maps. In \cite{ckt}, S. C. Chang, T. J. Kuo and J. Tie modified the method in \cite{y} and applied it to positive pseudohamonic functions on noncompact Sasakian $(2n+1)$-manifolds. They introduced a new auxiliary function and successfully dealt with the awkward term in Bochner-type formula. As a result, they obtained a sub-gradient estimate and Liouville theorem for positive pseudoharmonic functions.

In this paper, inspired by \cite{c, ckt}, we derive a sub-gradient estimate for pseudoharmonic maps from noncompact complete Sasakian manifolds which satisfies CR sub-Laplace comparison property (Theorem \ref{nsg}). Then we get the Liouville theorem for pseudoharmonic maps (Theorem \ref{cse}). In the Appendix, we apply the method to harmonic maps from noncompact complete Sasakian manifolds and derive a Reeb energy density estimate (Theorem \ref{csh}). From this estimate, we can prove Liouville theorem for harmonic maps on Sasakian manifolds.

\section{Basic Notions} \label{crb}
A smooth manifold $M$ of real dimension $(2n+1)$ is said to be a CR manifold, if there exists a smooth rank $n$ complex subbundle $T_{1,0} M \subset TM \otimes \mathbb{C}$ such that
$$
T_{1,0} M \cap T_{0,1} M =0
$$
and
$$
[\Gamma (T_{1,0} M), \Gamma (T_{1,0} M)] \subset \Gamma (T_{1,0} M)
$$
where $T_{0,1} M = \overline{T_{1,0} M}$ is the complex conjugate of $T_{1,0} M$. If $M$ is a CR manifold, then its Levi distribution is the real subbundle $HM = Re \: \{ T_{1,0}M \oplus T_{0,1}M \}$.
It carries a complex structure $J_b : HM \rightarrow HM$, which is given by $J_b (X+\overline{X})= \ii (X-\overline{X})$ for any $X \in T_{1,0} M$. Since $HM$ is naturally oriented by the complex structure, then $M$ is orientable if and only if
there exists a global nonvanishing 1-form $\theta$ such that $\theta (HM) =0 $.
Any such section $\theta$ is referred to as a pseudo-Hermitian structure on $M$. The Levi form $L_\theta $ is given by
$$L_\theta (Z,\overline{W} ) = - \ii d \theta (Z, \overline{W})  $$
for any $Z , W \in T_{1, 0} M$.
\begin{dfn}
An orientable CR manifold $M$ with a pseudo-Hermitian structure $\theta$, denoted by $(M, HM, J_b, \theta)$, is called a pseudo-Hermitian manifold. A pseudo-Hermitian manifold $(M, HM, J_b, \theta)$  is said to be a strictly pseudoconvex CR manifold if its Levi form $L_\theta$ is positive definite.
\end{dfn}

If $(M, HM, J_b, \theta)$ is strictly pseudoconvex, there exists a unique nonvanishing vector field $T$, transverse to $HM$, satisfying
$T \lrcorner \: \theta =1, \ T \lrcorner \: d \theta =0$. This vector field is called the characteristic direction of $(M, HM, J_b, \theta)$.
Define the bilinear form $G_\theta$ by
$$
G_\theta (X, Y)= d \theta (X, J_b Y)
$$
for $X, Y \in HM$. Since $L_\theta$ and $G_\theta$ coincide on $T_{1,0} M \otimes T_{0,1} M$, $G_\theta$ is also positive definite on $HM \otimes HM$. This allows us to define a Riemannian metric $g_\theta$ on $M$ by
$$g_\theta (X, Y) = G_\theta (\pi_H X, \pi_H Y)+ \theta(X) \theta (Y), \quad X, Y \in  TM$$
where $\pi_H : TM \rightarrow HM$ is the projection associated to the direct sum decomposition $TM = HM \oplus \mathbb{R} T$. This metric is usually called the Webster metric.

On a strictly pseudoconvex CR manifold, there exists a canonical connection preserving the complex structure and the Webster metric. Actually

\begin{prp} [\cite{dt}]
Let $(M, HM, J_b, \theta)$ be a strictly pseudoconvex CR manifold. Let $T$ be the characteristic direction and $J_b$ the complex structure in $HM$ (extending to an endomorphism of $TM$ by requiring that $J_b T=0$). Let $g_\theta$ be the Webster metric. Then there is a unique linear connection $\nabla$ on $M$ (called the Tanaka-Webster connection) such that:
\begin{enumerate}[(i)]
\item The Levi distribution HM is parallel with respect to $\nabla$.
\item $\nabla J_b=0$, $\nabla g_\theta=0$.
\item The torsion $T_\nabla$ of $\nabla$ satisfies $T_{\nabla} (X, Y)= 2 d \theta (X, Y) T  $ and $T_{\nabla} (T, J_b X) + J_b T_{\nabla} (T, X) =0 $ for any $X, Y \in HM$.
\end{enumerate}
\end{prp}

The pseudo-Hermitian torsion, denoted $\tau$, is the $TM$-valued 1-form defined by $\tau(X) = T_{\nabla} (T,X)$. Note that $\tau(T_{1,0} M) \subset T_{0,1} M$ and $\tau$ is $g_\theta$-symmetric (cf. \cite{dt}).

\begin{prp}[\cite{dt}]
If $(M, HM, J_b, \theta)$ is a strictly pseudoconvex CR manifold, the synthetic object $(J_b, -T, -\theta, g_\theta)$ is a contact metric structure on $M$. This contact metric structure is a Sasakian structure if and only if the pseudo-Hermitian torsion $\tau$ is zero.
\end{prp}

\begin{exm}[Heisenberg group]
The Heisenberg group $\mathbb{H}^n$ is obtained by $\mathbb{C}^n \times \mathbb{R}$ with the group law
$$
(z,t) \cdot (w,s) = (z+w, t+s+ 2 Im \langle z, w \rangle) .
$$
Let us consider the complex vector fields on $\mathbb{H}^n$,
$$
T_\alpha= \frac{\partial}{\partial z^\alpha} + \ii \overline{z^\alpha} \frac{\partial}{\partial t}
$$
where $\frac{\partial }{\partial z^\alpha} = \frac{1}{2} (\frac{\partial}{\partial x^\alpha} - \ii \frac{\partial}{\partial y^\alpha})$ and $z^\alpha= x^\alpha + \ii y^\alpha$.
The CR structure $T_{1,0} \mathbb{H}^n$ is spanned by $\{ T_1, \dots, T_n\}$.
There is a pseudo-Hermitian structure $\theta$ on $\mathbb{H}^n$ defined by
$$
\theta = d t+ 2 \sum_{\alpha =1}^n (x^\alpha d y^\alpha- y^\alpha d x^\alpha ).
$$
The Levi form $L_\theta= 2 \sum_{\alpha =1}^n d z^\alpha \wedge d z^{\ba} $ is positive definite, so $(\mathbb{H}^n, H \mathbb{H}^n, J_b, \theta)$ is a strictly pseudo-Hermitian  CR manifold. The characteristic direction is $T = \frac{\partial}{\partial t}$. Moreover, the Tanaka-Webster connection of $(\mathbb{H}^n, H \mathbb{H}^n, J_b, \theta)$ is flat. Hence the pseudo-Hermitian torsion is zero, and $(\mathbb{H}^n, H \mathbb{H}^n, J_b, \theta)$ is Sasakian (See \cite{dt} for details).
\end{exm}

Let $(M, HM, J_b, \theta)$ be a strictly pseudoconvex CR (2n+1)-manifold. Let $\{ Z_1, \dots, Z_n \}$ be a local orthonormal frame of $T_{1,0} M$ defined on the open set $U \subset M$ , and $\{ \theta^1, \dots \theta^n \}$ its dual coframe.
Then,
\begin{align*}
d \theta = 2 \ii \sum_{\alpha=1}^{n} \theta^\alpha \wedge \theta^{\ba} .
\end{align*}
Since $\tau(T_{1,0} M) \subset T_{0,1} M$, one can set $\tau Z_\alpha = A_{\alpha}^{\ \bb} Z_{\bb}$ for some local smooth functions $A_{\alpha}^{\ \bb} : U \rightarrow \mathbb{C}$. Denote by $\{ \omega_{\alpha}^{\ \beta} \}$ the Tanaka-Webster connection 1-forms with respect to the frame $\{ T_\alpha \}$, i.e. $\nabla Z_\alpha = \omega_{\alpha}^{\ \beta} \otimes Z_\beta$. Then the structure equations can be expressed as follows:
\begin{align} \label{se1}
d \theta^\beta  =  \theta^\alpha \wedge \omega_{\alpha}^{\ \beta} + \theta \wedge \tau^\beta, \quad
\tau_\alpha \wedge \theta^\alpha=0  , \quad
\omega_{\alpha}^{\ \beta} +  \omega_{\bb}^{\ \bar{\alpha}}=0
\end{align}
where $\tau^\alpha = A^\alpha_{\ \bb} \theta^{\bb} = A_{\ba \bb} \theta^{\bb}$ is a local 1-form.


In \cite{dt, w}, the authors showed that the curvature form of Tanaka-Webster connection $\Pi_{\beta}^{\ \alpha}  = d \omega_{\beta}^{\ \alpha} - \omega_{\beta}^{\ \gamma} \wedge \omega_{\gamma}^{\ \alpha}$ is given by
\begin{equation} \label{s3}
\Pi_{\beta}^{\ \alpha} = R_{\beta \  \mu \bar{\gamma}}^{\ \alpha} \theta^{\mu} \wedge \theta^{\bar{\gamma}} + W_{\beta \  \mu}^{\ \alpha} \theta^{\mu} \wedge \theta - W_{\  \beta \bar{\mu}}^{\alpha} \theta^{\bar{\mu}} \wedge \theta + 2 \ii \theta_{\beta} \wedge \tau^{\alpha} - 2 \ii \tau_{\beta} \wedge \theta^{\alpha}
\end{equation}
where $W_{\beta \ \mu}^{\ \alpha} = A_{\beta \mu, }^{\quad \  \alpha}$ and $W_{\ \beta \bar{\mu}}^\alpha = A_{\ \bar{\mu} , \beta}^\alpha$. In particular, $R_{\beta \ba \mu \bar{\gamma}}=R_{\mu \ba \beta \bar{\gamma}}$.
The pseudo-Hermitian $Ric$ tensor and the $Tor$ tensor on $T_{1,0} M$ are defined by
\begin{equation} \label{pric}
Ric(X,Y) = R_{\alpha \bb} X_{\ba} Y_{\beta} =R_{\alpha \bb \gamma \bar{\gamma}} X_{\ba} Y_{\beta}
\end{equation}
and
\begin{equation}
Tor(X,Y) = \ii (X_\alpha Y_\beta A_{\ba \bb}- X_{\ba} Y_{\bb} A_{\alpha \beta})
\end{equation}
for $X=X_{\ba} Z_{\alpha} \in M, \ Y=Y_{\bb} Z_{\beta} \in M$. 


Assume that $(N,h)$ is a Riemannian manifold. Let $\{ \xi_i \}$ be a local orthonormal frame of $TN$, and $\{ \sigma^i \}$ its dual coframe. Denote by $\{ \eta^{\ i}_j \}$ the connection 1-forms of the Levi-Civita connection $\hat{\nabla}$ on $N$, i.e. $\hat{\nabla} \xi_i = \eta_i^{\ j} \otimes \xi_j$. Then we have the structure equations
\begin{align} \label{se2}
d \sigma^i  =  \sigma^j \wedge \eta^{\ i}_j, \quad
d \eta^{\ i}_j  =  \eta^{\ l}_j \wedge \eta^{\ i}_l + \Omega^{\ i}_j, \quad
\Omega^{\ i}_j  =  \frac{1}{2} \hat{R}_{j \ kl}^{\ i} \sigma^k \wedge \sigma^l ,
\end{align}
where $\hat{R}$ is the curvature of Levi-Civita connection $\hat{\nabla}$ in $(N,h)$.

Suppose that $(M, HM, J_b, \theta)$ is a strictly pseudoconvex CR (2n+1)-manifold and $\nabla$ is its Tanaka-Webster connection. Let $f:M \rightarrow N$ be a smooth map and $f^* TN$ the pullback bundle.
Denote
\begin{align} \label{e1}
\begin{gathered}
d_b f =\pi_H df= f^i_\alpha \theta^\alpha \otimes \xi_i + f^i_{\bar{\alpha}} \theta^{\bar{\alpha}} \otimes \xi_i \ \in \Gamma (T^*M \otimes f^*TN) , \\
f_0 = df(T) = f_0^i \: \xi_i \ \in \Gamma (f^*TN) .
\end{gathered}
\end{align}
Let $\nabla^f$ be the pullback connection in $f^* TN$ induced by the Levi-Civita connection of $(N, h)$.
Then we can determine a connection $\nabla^f$ in $T^* M \otimes f^* TN$ by
$$ \nabla_X^f (\omega \otimes \xi) = \nabla_X \omega \otimes \xi + \omega \otimes \nabla_X^f \: \xi$$
for any $X \in \Gamma(TM)$, $\omega \in \Gamma(T^*M)$ and $\xi \in \Gamma(f^*TN)$.
Under the local frame $\{ \theta, \theta^\alpha, \theta^{\ba} \}$ and $\{ \xi_i \}$, the tensor $\nabla^f df$  can be expressed  by:
\begin{align}
\nabla^f d f = & \  f^i_{\alpha \beta} \theta^\alpha \otimes \theta^\beta \otimes \xi_i +
f^i_{\ba \beta} \theta^{\ba} \otimes \theta^{\beta} \otimes \xi_i +
f^i_{\alpha \bb} \theta^{\alpha} \otimes \theta^{\bb} \otimes \xi_i \nonumber \\
& \  +f^i_{\ba \bb} \theta^{\ba} \otimes \theta^{\bb} \otimes \xi_i + f^i_{0 \alpha} \theta \otimes \theta^\alpha \otimes \xi_i + f^i_{\alpha 0} \theta^\alpha \otimes \theta \otimes \xi_i  \nonumber \\
& \ + f^i_{0 \ba } \theta \otimes \theta^{\ba} \otimes \xi_i + f^i_{\ba 0} \theta^{\ba} \otimes \theta \otimes \xi_i +f^i_{00} \theta \otimes \theta \otimes \xi_i . \label{se3}
\end{align}
Denote by $\nabla^f_b d_b f$ the restriction of $\nabla^f df$ to $HM \times HM$.
Throughout the paper, the  Einstein summation convention is used (except in the inequality \eqref{bfp1}) and the ranges of indices are
$$
\alpha, \beta, \gamma, \mu ,\dots  \in \{ 1, \dots, n \}, \quad i, j, k, l , \dots \in \{ 1, \dots, \mbox{dim N} \}
$$
where dim $M = 2n+1$.

\begin{dfn}[\cite{dt}]
Let us consider the $f-$tensor field on $M$ given by
$$\tau(f;\theta,\hat{\nabla})=trace_{G_\theta} ( \nabla^f_b d_b f )\in \Gamma (f^* TN) .$$
We say that $f$ is pseudoharmonic, if $\tau(f;\theta,\hat{\nabla})=0$.
\end{dfn}
It is known that pseudoharmonic maps are the critical points of the following energy functional (cf. \cite{dt}):
$$E_\Omega (f)= \frac{1}{2} \int_\Omega trace_{G_\theta} (\pi_H f^*h) \ \theta \wedge (d\theta)^n$$
for any compact domain $\Omega \subset \subset M$.
With respect to the local frame $\{ \theta, \theta^\alpha, \theta^{\ba} \}$ and $\{ \xi_i \}$, we have
\begin{align} \label{ph}
\tau(f;\theta,\hat{\nabla})   = (f^i_{\alpha \ba} + f^i_{\ba \alpha}) \xi_i .
\end{align}

\section{Bochner-Type formulas}
In \cite{g}, A. Greenleaf obtained the commutation relations of smooth functions and established Bochner-type formulas of pseudoharmonic functions. In \cite{lee}, John M. Lee derived the commutation relations of $(1,0)$-forms. We shall need the commutation relations of various covariant derivatives of smooth maps and Bochner-type formulas of pseudoharmonic maps.

\begin{lem} \label{cc}
Let $f: M\rightarrow N$ be a smooth map. The covariant derivatives of $df$ satisfy the following commutation relations:
\begin{align}
f^i_{\alpha \beta}  = & f^i_{\beta \alpha}, \label{s1} \\
f^i_{\alpha \bb} -f^i_{\bb \alpha}  = & 2 \ii f^i_0 \delta_{\alpha \bb}, \label{s6} \\
f^i_{0 \alpha} -f^i_{\alpha 0} = & f^i_{\bb} A^{\bb}_{\ \alpha} , \label{s2}
\end{align}
and
\begin{align}
f^i_{\alpha \beta \gamma} - f^i_{\alpha \gamma \beta} = &2 \ii f^i_{\beta} A_{\alpha \gamma} -2 \ii f^i_\gamma A_{\alpha \beta} -f^j_\alpha f^k_\beta f^l_\gamma \hat{R}_{j \ kl}^{\ i}, \label{s4} \\
f^i_{\alpha \bb \bg}- f^i_{\alpha \bg \bb} = & 2 \ii f^i_\mu A^{\mu}_{\ \bg} \delta_{\alpha \bb}- 2 \ii f^i_\mu A^{\mu}_{\ \bb}  \delta_{\alpha \bg}- f^j_{\alpha} f^k_{\bb} f^l_{\bg} \hat{R}_{j \ kl}^{\ i}, \\
f^i_{\alpha \beta \bg} - f^i_{\alpha \bg \beta} =& f^i_\mu R_{\alpha \ \beta \bg}^{\ \mu} + 2 \ii f^i_{\alpha 0} \delta_{\beta \bg} - f^j_{\alpha} f^k_{\beta} f^l_{\bg} \hat{R}_{j \ kl}^{\ i}, \\
f^i_{\alpha \beta 0} - f^i_{\alpha 0 \beta}  = & f^i_\gamma A_{\alpha \beta, }^{\quad \ \gamma}-f^i_{\alpha \bg} A^{\bg}_{\ \beta}- f^j_{\alpha} f^k_{\beta} f^l_0 \hat{R}_{j \ kl}^{\ i}, \\
f^i_{\alpha \bb 0} -f^i_{\alpha 0 \bb}  =& -f^i_\gamma A^{\gamma}_{\ \bb,\alpha}- f^i_{\alpha \gamma} A^{\gamma}_{\ \bb} -f^j_{\alpha} f^k_{\bb} f^l_0 \hat{R}_{j \ kl}^{\ i}. \label{s5}
\end{align}
\end{lem}

\begin{proof}
The identities \eqref{e1} imply
\begin{align} \label{e2}
f^* \sigma^i = f^i_\alpha \theta^\alpha + f^i_{\ba} \theta^{\ba} + f^i_0 \theta .
\end{align}
We take the exterior derivative of \eqref{e2} and use the structure equations \eqref{se1}, \eqref{se2} to get
\begin{align}
0=& (d f^i_\alpha -f^i_\beta \omega^{\  \beta}_\alpha + f^j_\alpha \tilde{\eta}^{\ i}_j ) \wedge \theta^\alpha
+ (d f^i_{\ba} -f^i_{\bb} \omega^{\ \bb}_{\ba} + f^j_{\ba} \tilde{\eta}^{\ i}_j ) \wedge \theta^{\ba} \nonumber \\
& + (d f^i_0 + f^j_0 \tilde{\eta}^{\ i}_j) \wedge \theta + f^i_\alpha \theta \wedge \tau^\alpha + f^i_{\ba} \theta \wedge \tau^{\ba} +2  \ii f^i_0 \theta^\beta \wedge \theta^{\bb} , \label{se4}
\end{align}
where $\tilde{\eta}^{\ i}_j = f^* \eta^{\ i}_j$. 
On the other hand, the second-order covariant derivatives satisfy
\begin{align}
d f^i_\alpha -f^i_\beta \omega^{\  \beta}_\alpha + f^j_\alpha \tilde{\eta}^{\ i}_j &= f^i_{\alpha \beta} \theta^\beta + f^i_{\alpha \bb} \theta^{\bb} +f^i_{\alpha 0} \theta , \label{se5} \\
d f^i_{\ba} -f^i_{\bb} \omega^{\ \bb}_{\ba} + f^j_{\ba} \tilde{\eta}^{\ i}_j &= f^i_{\ba \beta} \theta^\beta + f^i_{\ba \bb} \theta^{\bb} +f^i_{\ba 0} \theta , \\
d f^i_0 + f^j_0 \tilde{\eta}^{\ i}_j &= f^i_{0 \beta} \theta^\beta + f^i_{0 \bb} \theta^{\bb} +f^i_{0 0} \theta .
\end{align}
Substituting the above three equations into \eqref{se4} and using $\tau^\alpha = A^\alpha_{\ \bb} \theta^{\bb}$, we obtain
\begin{align*}
0= & f^i_{\alpha \beta} \theta^{\beta} \wedge \theta^{\alpha} +f^i_{\ba \bb} \theta^{\bb} \wedge \theta^{\ba} + (f^i_{\alpha \bb} -f^i_{\bb \alpha} - 2\ii f^i_0 \delta_{\alpha \bb} ) \theta^{\bb} \wedge \theta^\alpha \\
& + (f^i_{\alpha 0} - f^i_{0 \alpha} + f^i_{\bb} A^{\bb}_{\ \alpha}) \theta \wedge \theta^\alpha + (f^i_{\ba 0} - f^i_{0 \ba} + f^i_{\beta} A^{\beta}_{\ \ba}) \theta \wedge \theta^{\ba} .
\end{align*}
which (by comparing types) yields \eqref{s1}-\eqref{s2}. To prove the next five equations, we differentiate \eqref{se5} and use the structure equations again. Then we obtain
\begin{align}
0= & (d f^i_{\alpha \beta} - f^i_{\alpha \gamma} \omega^{\ \gamma}_{\beta} - f^i_{\gamma \beta} \omega_{\alpha}^{\ \gamma} + f^j_{\alpha \beta}  \tilde{\eta}^{\ i}_j) \wedge \theta^\beta \nonumber \\
&+(d f^i_{\alpha \bb} - f^i_{\alpha \bg} \omega^{\ \bg}_{\bb} - f^i_{\gamma \bb} \omega_{\alpha}^{\ \gamma} + f^j_{\alpha \bb} \tilde{\eta}^{\ i}_j) \wedge \theta^{\bb} \nonumber \\
&+ (d f^i_{\alpha 0} - f^i_{\beta 0} \omega_{\alpha}^{\ \beta} + f^j_{\alpha 0}  \tilde{\eta}^{\ i}_j) \wedge \theta + f^i_\beta \Pi^{\ \beta}_{ \alpha}- f^j_\alpha f^*(\Omega^i_j) \nonumber \\
& + f^i_{\alpha \beta} A^{\beta}_{\ \bg} \theta \wedge \theta^{\bg} + f^i_{\alpha \bb} A^{\bb}_{\ \gamma} \theta \wedge \theta^\gamma+ 2 \ii f^i_{\alpha 0} \theta^\beta \wedge \theta_\beta . \label{se6}
\end{align}
Since the third-order covariant derivatives of $f$ is given by
\begin{align*}
d f^i_{\alpha \beta} - f^i_{\alpha \gamma} \omega^{\ \gamma}_{\beta} - f^i_{\gamma \beta} \omega_{\alpha}^{\ \gamma} + f^j_{\alpha \beta}  \tilde{\eta}^{\ i}_j =& f^i_{\alpha \beta \gamma} \theta^{\gamma} + f^i_{\alpha \beta \bg} \theta^{\bg} + f^i_{\alpha \beta 0} \theta , \\
d f^i_{\alpha \bb} - f^i_{\alpha \bg} \omega^{\ \bg}_{\bb} - f^i_{\gamma \bb} \omega_{\alpha}^{\ \gamma} + f^j_{\alpha \bb}  \tilde{\eta}^{\ i}_j =& f^i_{\alpha \bb \gamma} \theta^{\gamma} + f^i_{\alpha \bb \bg} \theta^{\bg} + f^i_{\alpha \bb 0} \theta , \\
d f^i_{\alpha 0} - f^i_{\beta 0} \omega_{\alpha}^{\ \beta} + f^j_{\alpha 0}  \tilde{\eta}^{\ i}_j = & f^i_{\alpha 0 \gamma} \theta^{\gamma} + f^i_{\alpha 0 \bg} \theta^{\bg} + f^i_{\alpha 0 0} \theta ,
\end{align*}
we can substitute them into \eqref{se6} and use \eqref{s3}, \eqref{se2} to obtain
\begin{align*}
0&= \sum_{\gamma < \beta} (f^i_{\alpha \beta \gamma} - f^i_{\alpha \gamma \beta} - 2\ii f^i_{\beta} A_{\alpha \gamma} + 2 \ii f^i_\gamma A_{\alpha \beta} + f^j_\alpha f^k_\beta f^l_\gamma \hat{R}_{j \ kl}^{\ i}) \theta^\gamma \wedge \theta^\beta \\
&+ \sum_{\gamma , \beta} (f^i_{\alpha \beta \bg} - f^i_{\alpha \bg \beta} - f^i_\mu R_{\alpha \ \beta \bg}^{\ \mu} - 2 \ii f^i_{\alpha 0} \delta_{\beta \bg} + f^j_{\alpha} f^k_{\beta} f^l_{\bg} \hat{R}_{j \ kl}^{\ i}) \theta^{\bg} \wedge \theta^{\beta} \\
&+ \sum_{\gamma < \beta} (f^i_{\alpha \bb \bg}- f^i_{\alpha \bg \bb} -2 \ii f^i_\mu A^{\mu}_{\ \bg} \delta_{\alpha \bb}+ 2\ii f^i_\mu A^{\mu}_{\ \bb}  \delta_{\alpha \bg}+ f^j_{\alpha} f^k_{\bb} f^l_{\bg} \hat{R}_{j \ kl}^{\ i}) \theta^{\bg} \wedge \theta^{\bb} \\
&+ \sum_{\beta} (f^i_{\alpha \beta 0} - f^i_{\alpha 0 \beta}  - f^i_\gamma A_{\alpha \beta, }^{\quad \ \gamma}+f^i_{\alpha \bg} A^{\bg}_{\ \beta}+ f^j_{\alpha} f^k_{\beta} f^l_0 \hat{R}_{j \ kl}^{\ i}) \theta \wedge \theta^\beta \\
&+ \sum_{\beta} (f^i_{\alpha \bb 0} -f^i_{\alpha 0 \bb}  +f^i_\gamma A^{\gamma}_{\ \bb,\alpha}+ f^i_{\alpha \gamma} A^{\gamma}_{\ \bb} +f^j_{\alpha} f^k_{\bb} f^l_0 \hat{R}_{j \ kl}^{\ i}) \theta \wedge \theta^{\bb} .
\end{align*}
which (by comparing types) yields \eqref{s4}-\eqref{s5}.
\end{proof}

Before introducing the Bochner-type formulas, we recall a property of the sub-Laplace operator $\triangle_b$ (cf. \cite{dt}). If $u$ is a $C^2$ function on $M$, then $\triangle_b u = trace_{G_\theta} (\nabla_b d_b u)$. With respect to the local orthonormal frame $\{ T, Z_\alpha, Z_{\ba} \}$, we have $\triangle_b u = u_{\alpha \ba}+ u_{\ba \alpha}$.

\begin{lem} \label{c1}
For any smooth map $f: M \rightarrow N$, we have
\begin{align}
\frac{1}{2} \triangle_b |d_b f|^2 =& \ |\nabla_b^f d_b f|^2 +\langle \nabla_b^f  \tau (f;\theta,\hat{\nabla}), d_b f\rangle - 4 \langle  d_b f \circ J_b, \nabla_b^f f_0\rangle \nonumber\\
& \ + (2Ric - 2 (n-2)Tor) (f^i_{\bb} Z_\beta, f^i_{\ba} Z_\alpha)  \nonumber\\
& \ +2 (f^i_{\bar{\alpha}} f^j_{\beta} f^k_{\bar{\beta}} f^l_{\alpha} \hat{R}^{\ i}_{j\ kl}
+ f^i_{\alpha} f^j_{\beta} f^k_{\bar{\beta}} f^l_{\bar{\alpha}} \hat{R}^{\ i}_{j\ kl}) , \label{bf1}\\
\frac{1}{2} \triangle_b |d f(T)|^2  =&  \ | \nabla_b^f f_0 |^2 + \langle df(T), \nabla_T^f \; \tau(f;\theta,\hat{\nabla})\rangle  + 2 f^i_0 f^j_{\alpha} f^k_{\ba} f^l_0 \hat{R}_{j \ kl}^{\ i} \nonumber\\
 &\  +2( f^i_0 f^i_\beta A_{\bb \ba , \alpha} + f^i_0 f^i_{\bb} A_{\beta \alpha , \ba} + f^i_0 f^i_{\bb \ba} A_{\beta \alpha} +f^i_0 f^i_{\beta \alpha} A_{\bb \ba} ) \label{bf2}
\end{align}
where $\nabla_b^f  \tau (f;\theta,\hat{\nabla})$ and $\nabla_b^f f_0$ are the restriction of $\nabla^f  \tau (f;\theta,\hat{\nabla})$ and $\nabla^f f_0$ to $HM$.
\end{lem}

\begin{proof}
Since $|d_b f|^2  = 2 f^i_\alpha f^i_{\ba}$, we have
\begin{align*}
\frac{1}{2} \triangle_b |d_b f|^2 =& (f^i_\alpha f^i_{\ba})_{\beta \bb}+ (f^i_\alpha f^i_{\ba})_{\bb \beta} \\
= & 2(f^i_{\alpha \beta} f^i_{\ba \bb} + f^i_{\alpha \bb} f^i_{\ba \beta}) + f^i_{\ba} f^i_{\alpha \beta \bb} + f^i_{\alpha} f^i_{\ba \beta \bb} + f^i_{\alpha} f^i_{\ba \bb \beta} + f^i_{\ba} f^i_{\alpha \bb \beta} \\
= & |\nabla^f_b d_b f|^2 + f^i_{\ba} f^i_{\alpha \beta \bb} + f^i_{\alpha} f^i_{\ba \beta \bb} + f^i_{\alpha} f^i_{\ba \bb \beta} + f^i_{\ba} f^i_{\alpha \bb \beta} .
\end{align*}
Lemma \ref{cc} implies
\begin{align*}
f^i_{\alpha \beta \bb} = & f^i_{\beta \alpha \bb} = f^i_{\beta \bb \alpha  } +f^i_\mu R_{\beta \ \alpha \bb}^{\ \mu} + 2 \ii f^i_{\beta 0} \delta_{\alpha \bb} - f^j_{\beta} f^k_{\alpha} f^l_{\bb} \hat{R}_{j \ kl}^{\ i} \\
=& f^i_{\beta \bb \alpha  } +f^i_\mu R_{\beta \ \alpha \bb}^{\ \mu} + 2 \ii (f^i_{0 \beta}- f^i_{\bar{\mu}} A^{\bar{\mu}}_{\ \beta}) \delta_{\alpha \bb} - f^j_{\beta} f^k_{\alpha} f^l_{\bb} \hat{R}_{j \ kl}^{\ i}, \\
f^i_{\ba \beta \bb} = & (f^i_{\beta \ba} - 2 \ii f^i_0 \delta_{\beta \ba})_{\bb} = f^i_{\beta \ba \bb} - 2 \ii f^i_{0 \bb} \delta_{\beta \ba} \\
= & f^i_{\beta \bb \ba}+ 2 \ii f^i_\mu A^{\mu}_{\ \bb} \delta_{\beta \ba}- 2 \ii f^i_\mu A^{\mu}_{\ \ba}  \delta_{\beta \bb}- f^j_{\beta} f^k_{\ba} f^l_{\bb} \hat{R}_{j \ kl}^{\ i}- 2 \ii f^i_{0 \bb} \delta_{\beta \ba} .
\end{align*}
Substituting them into the previous identity, we obtain
\begin{align*}
\frac{1}{2} \triangle_b |d_b f|^2 =& |\nabla^f_b d_b f|^2 + f^i_{\ba} (f^i_{\beta \bb} + f^i_{\bb \beta})_{\alpha} + f^i_{\alpha} (f^i_{\beta \bb} + f^i_{\bb \beta})_{\ba}  \\
&+ 2 f^i_{\ba} f^i_{\mu} R_{\alpha \bar{\mu} \beta \bb} - 2 \ii (n-2) (f^i_{\alpha} f^i_\mu A_{\bar{\mu} \ba}- f^i_{\ba} f^i_{\bar{\mu}} A_{\mu \alpha}) \\
& +2 (f^i_{\bar{\alpha}} f^j_{\beta} f^k_{\bar{\beta}} f^l_{\alpha} \hat{R}^{\ i}_{j\ kl}
+ f^i_{\alpha} f^j_{\beta} f^k_{\bar{\beta}} f^l_{\bar{\alpha}} \hat{R}^{\ i}_{j\ kl}) + 4 \ii (f^i_{\ba} f^i_{0 \alpha} - f^i_{\alpha} f^i_{0 \ba}) .
\end{align*}
By the identity $\langle  d_b f \circ J_b, \nabla_b^f f_0\rangle= \ii (f^i_{\alpha} f^i_{0 \ba} -f^i_{\ba} f^i_{0 \alpha}) $, we get \eqref{bf1}. The proof of \eqref{bf2} is similar.
\end{proof}

\begin{lem} \label{bte}
Let $(M, HM, J_b, \theta)$ be a $(2n+1)$-Sasakian manifold with
\begin{equation} \label{ric3}
Ric(X,X) \geq -k |X|^2
\end{equation}
for all $X \in T_{1,0} M$, and some $k \geq 0$. Suppose that $(N,h)$ is a Riemannian manifold with nonpositive sectional curvature. If $f: M \rightarrow N$ is a pseudoharmonic map, then for any $\nu>0$, we have
\begin{equation}
\triangle_b |d_b f|^2 \geq |\nabla_b^f d_b f|^2 + 2n |f_0|^2 -  \left( 2k+ \frac{32}{\nu } \right) |d_b f|^2 - \frac{1}{2} \nu  |\nabla_b^f f_0|^2 \label{nsm3}
\end{equation}
and
\begin{equation} \label{nsm4}
\triangle_b |f_0|^2 \geq 2 |\nabla_b^f f_0|^2 .
\end{equation}
\end{lem}

\begin{proof}
Since $f$ is pseudoharmonic, by definition we have $\tau(f;\theta,\hat{\nabla}) =0$. Because $(M, HM, J_b, \theta)$ is Sasakian, the tensor $Tor=0$. Hence, by the assumption on the pseudo-Hermitian Ricci curvature, \eqref{bf1} becomes
\begin{align}
 \triangle_b |d_b f|^2 =& 2|\nabla_b^f d_b f|^2 - 8 \langle  d_b f \circ J_b, \nabla_b^f f_0\rangle - 2 k |d_b f|^2 \nonumber\\
& \ +4 (f^i_{\bar{\alpha}} f^j_{\beta} f^k_{\bar{\beta}} f^l_{\alpha} \hat{R}^{\ i}_{j\ kl}
+ f^i_{\alpha} f^j_{\beta} f^k_{\bar{\beta}} f^l_{\bar{\alpha}} \hat{R}^{\ i}_{j\ kl}). \label{e3}
\end{align}
Using the commutation relation \eqref{s6}, we can estimate
\begin{align}
|\nabla_b^f d_b f |^2= & 2 \: \sum_{\alpha, \beta =1}^n (f^i_{\ba \beta} f^i_{\alpha \bb} + f^i_{\alpha \beta} f^i_{\ba \bb}) \geq  2 \: \sum_{\alpha=1}^n f^i_{\alpha \ba} f^i_{\ba \alpha} \nonumber \\
= &  \frac{1}{2} \: \sum_{\alpha=1}^n [|f^i_{\alpha \ba}+ f^i_{\ba \alpha}|^2+ |f^i_{\alpha \ba}- f^i_{\ba \alpha}|^2]\
\geq \frac{1}{2} \: \sum_{\alpha=1}^n  |f^i_{\alpha \ba}- f^i_{\ba \alpha}|^2 \nonumber \\
= & 2n |f_0|^2 . \label{bfp1}
\end{align}
The second term of the right side of \eqref{e3} can be controlled by the Schwarz inequality
\begin{equation} \label{bfp2}
- 8 \langle d_b f  \circ J_b, \nabla_b^f f_0 \rangle \geq -   \frac{32}{\nu}  |d_b f|^2 - \frac{1}{2} \nu |\nabla_b^f f_0|^2 .
\end{equation}
To deal with the last term of \eqref{e3}, we set $e_\alpha = Re \; df(Z_\alpha)$ and $ e_{\alpha}' = Im  \; df(Z_\alpha) $. Then
\begin{align}
Last\ term \  of \  \eqref{e3} =& \ 4 \langle \hat{R}( df(Z_{\bb}), df(Z_{\ba}) ) df(Z_\beta), df(Z_\alpha)\rangle  \nonumber \\
& \ \ + 4\langle  \hat{R}( df(Z_{\bar{\beta}}) , df(Z_\alpha) ) df(Z_\beta), df(Z_{\bar{\alpha}})\rangle  \nonumber \\
=& \  -4( \langle \hat{R}( e_\alpha, e_\beta ) e_\beta, e_\alpha\rangle  + \langle \hat{R}( e_\alpha, e_{\beta}' ) e_{\beta}', e_\alpha\rangle \nonumber  \\
& \ \ + \langle \hat{R}( e_{\alpha}', e_\beta ) e_\beta, e_{\alpha}'\rangle  +  \langle \hat{R}( e_{\alpha}', e_{\beta}' ) e_{\beta}', e_{\alpha}'\rangle  )  \nonumber \\
\geq & \  0 \label{bfp3}
\end{align}
where we have used the assmuption that the sectional curvature of $N$ is nonpositive. Substituting \eqref{bfp1}, \eqref{bfp2} and \eqref{bfp3} into \eqref{e3}, we get \eqref{nsm3}.

Observe that
\begin{align}
f^i_0 f^j_{\alpha} f^k_{\ba} f^l_0 \hat{R}_{j \ kl}^{\ i} = & \langle \hat{R}( df(Z_{\ba}), df(T) ) df(Z_\alpha) , df(T)\rangle  \nonumber\\
= & -(\langle \hat{R}(e_\alpha,e_0)e_0,e_\alpha\rangle  + \langle \hat{R}(e_{\alpha}',e_0)e_0,e_{\alpha}'\rangle ) \nonumber \\
\geq & 0  . \label{btep}
\end{align}
Then \eqref{nsm4} can be easily proved from \eqref{bf2} and \eqref{btep}. 
\end{proof}

%

From now on, we assume that $(N,h)$ is a simply connected Riemannian manifold with nonpositive sectional curvature. Let $\rho$ be the distance to a fixed point $y_0 \in N$. Then $\rho^2$ is smooth on $N$. By the Hession comparison theorem, we have
$$Hess(\rho^2) \geq 2h . $$
For any smooth map $f: M \rightarrow N$, the chain rule gives that
\begin{align*}
\triangle_b (\rho^2 \circ f)  =  d \rho^2(\tau(f;\theta,\hat{\nabla})) + trace_{G_\theta} \: Hess(\rho^2)(d_b f,d_b f).
\end{align*}
Therefore, we can conclude that if $f$ is pseudoharmonic, then
\begin{equation} \label{nd1}
\triangle_b (\rho^2 \circ f) \geq 2|d_b f|^2.
\end{equation}

\section{Cannot-Carath\'eodory distance}
As known, the maximum principle is an important tool to obtain pointwise estimates for solutions of geometric PDEs.
In order to use it in Sasakian manifolds, we need some special exhaustion function to construct a cutoff function. A natural choice is the Carnot-Carath\'eodory distance function.
\begin{dfn} \label{ccd}
Let $(M, HM, J_b, \theta)$ be a strictly pseudoconvex CR manifold. A piecewise $C^1$-curve $\gamma : [0,1] \rightarrow M$ is said to be horizontal if $\gamma' (t) \in HM$ whenever $\gamma' (t)$ exists. The length of $\gamma$ is given by $$ l(\gamma) =\int^1_0 |\gamma'|^{1/2}_{G_\theta} dt . $$
We define the Cannot-Carath\'eodory distance between two points $p,q \in M$ by
$$ d_c(p,q) = inf \{ l(\gamma) | \: \gamma \in C_{p,q} \} $$
where $C_{p,q}$ is the set of all horizontal curves joining $p$ and $q$. We say that $(M, HM, J_b, \theta)$ is complete if it is complete as a metric space. A horizontal curve $\gamma: [0,1] \rightarrow M$ is called length minimizing geodesic if $l(\gamma) =d_c(\gamma(0), \gamma (1))$. Fix $x_0 \in M$, and set $r(x) = d_c(x_0,x)$. The Carnot-Carath\'eodory ball of radius $R$ centered at $x_0$ is denoted by $B_R(x_0) = \{ x \in M | \: r(x) <R \}$.
\end{dfn}

In \cite{s}, R. Strichartz pointed out that if $(M, HM, J_b, \theta)$ is complete, then for any $x_0, x \in M$, there exists at least one length minimizing geodesic $\gamma : [0,1] \rightarrow M$ joining $x_0$ and $x$. Moreover, $\gamma$ can extend to $(-\infty, \infty)$.
We say that $x$ is a cut point of $x_0$, if for any $\epsilon > 0$, $\gamma |_{[0,1+ \epsilon]}$ is no longer a length minimizing geodesic joining $x_0$ and $\gamma (1+ \epsilon)$. The set of all cut points of $x_0$, denoted by $cut(x_0)$, is called the cut locus of $x_0$.
Theorem 1.2 and Proposition 1.2 in \cite{a} assert that the Cannot-Carath\'eodory distance $r$ to a reference point $x_0 $ is smooth on $M \setminus ( cut(x_0) \cup \{ x_0 \} )$.

\begin{dfn}[\cite{ckt}] \label{crcp}
Let $(M, HM, J_b, \theta)$ be a noncompact complete Sasakian $(2n+1)$-manifold with
\begin{equation*}
Ric(X,X) \geq -k |X|^2
\end{equation*}
for all $X \in T_{1,0} M$ and some $k \geq 0$.
We say that $(M, HM, J_b, \theta)$ satisfies CR sub-Laplace comparison property relative to a point $x_0 \in M$, if there exists a positive constant $C_1$
such that the Carnot-Carath\'eodory distance $r$ to $x_0$ satisfies
\begin{equation}
\triangle_b r \leq C_1 (\frac{1}{r}+\sqrt{k})
\end{equation}
on $M \setminus ( cut(x_0) \cup \{ x_0 \} )$ and where $r \geq 1$.
\end{dfn}

\begin{prp}[\cite{ctw}]  \label{hp1}
There exists a positive constant $C_1'$ on Heisenberg group $(\mathbb{H}^n, H \mathbb{H}^n, J_b, \theta)$ such that
\begin{eqnarray}
\triangle_b r  \leq  \frac{C_1'}{r} \label{h1}
\end{eqnarray}
on $M \setminus ( cut(o) \cup \{ o \} )$. Here $r$ is the Carnot-Carath\'eodory distance to the origin $o$.
\end{prp}
Since the pseudohermitian torsion and the pseudohermitian Ricci curvature of Heisenberg group  $(\mathbb{H}^n, H \mathbb{H}^n, J_b, \theta)$ are both zero, Proposition \ref{hp1} asserts that the CR sub-Laplace comparison property holds on Heisenberg group. 

\section{Sub-Gradient Estimate For Pseudoharmonic Map} \label{slp}
In this section, we will obtain a sub-gradient estimate for pseudoharmonic maps.
Let $(M, HM, J_b, \theta)$ be a noncompact complete $(2n+1)$-Sasakian manifold with CR sub-Laplace comparison property relative to a point $x_0 \in M$ and
\begin{equation*}
 Ric(X,X) \geq -k |X|^2
\end{equation*}
for all $X \in T_{1,0} M$ and some $k\geq 0$. Suppose that $(N,h)$ is a simply connected Riemannian manifold with nonpositive sectional curvature. We consider a pseudoharmonic map $f: M \rightarrow N$. Let $\rho $ be the Riemannian distance to $y_0 = f(x_0)$.

We choose a function  $\psi \in C^{\infty} ([0, \infty))$ with the property that
$$ \psi |_{[0,1]} =1, \quad \psi |_{[2, \infty)}=0, \quad -C_2 \: |\psi|^{\frac{1}{2}} \leq \psi ' \leq 0, \quad |\psi ''| \leq C_2 . $$
Let $R > 1 $ be fixed. By CR sub-Laplacian comparison property, the cutoff function $\eta = \psi(\frac{r}{R}) $ satisfies:
\begin{equation}
\begin{gathered}
\eta^{-1} |d_b \eta|^2  \leq   \frac{C_2'}{R^2}  \\
\triangle_b \eta  =  \frac{\psi ''}{R^2} |d_b r|^2 + \frac{\psi '}{R} \triangle_b r \geq - C_2' \: \left( \frac{1}{R^2}+\frac{\sqrt{k}}{R} \right) \label{nseta}
\end{gathered}
\end{equation}
on $M \setminus ( cut(x_0) \cup \{ x_0 \} )$. Here $C_2'$ depends only on $C_2$ and $C_1$. Denote $b_R =2 \: sup \: \{ \rho \circ f(x) | x \in B_{2R} (x_0) \}$. We construct a smooth function $F(x): B_{2R} (x_0)  \rightarrow \mathbb{R}$ by
\begin{equation} \label{nf}
F(x)= \frac{|d_b f|^2 + \mu \eta |f_0|^2}{ b_R^2 - \rho^2 \circ f} (x) .
\end{equation}
The positive coefficient $\mu$ will be determined later. 
\begin{lem} \label{l1}
If $r$ is smooth at $x \in B_{2R} (x_0)$ and $(\eta F) (x) \neq 0$, then at $x$, we have
\begin{align}
& \triangle_b (|d_b f|^2 + \mu \eta |f_0|^2) \geq  \frac{1}{24} \frac{|d_b (|d_b f|^2 + \mu \eta |f_0|^2)|^2}{|d_b f|^2 + \mu \eta |f_0|^2}  \nonumber\\
&  \quad \quad \qquad + \left[ 2n - 6 \mu C_2' ( \frac{1}{R^2}+\frac{\sqrt{k}}{R}) \right] |f_0|^2 -  32 \left( k+ \frac{1}{\mu \eta} \right) |d_b f|^2 . \label{nsm6}
\end{align}
\end{lem}

\begin{proof}
First we compute 
\begin{align*}
\triangle_b (\mu \eta |f_0|^2) = & \mu[(\triangle_b \eta) |f_0|^2 + 2 \langle d_b \eta, d_b |f_0|^2\rangle  + \eta \: \triangle_b |f_0|^2]  \\
= & \mu[(\triangle_b \eta) |f_0|^2 + 2 \langle d_b \eta, 2 \langle \nabla_b^f f_0, f_0\rangle_{f^*TN} \rangle  + \eta \: \triangle_b |f_0|^2]   \\
\geq & \mu[(\triangle_b \eta) |f_0|^2 - \eta |\nabla_b^f f_0|^2 - 4 |f_0|^2 \eta^{-1} |d_b \eta|^2 + \eta \: \triangle_b |f_0|^2] \\
\geq & \  \mu \eta |\nabla_b^f f_0|^2- 5 \mu C_2' \: \left(\frac{1}{R^2}+\frac{\sqrt{k}}{R} \right) |f_0|^2 .
\end{align*}
The last inequality is due to \eqref{nsm4} and \eqref{nseta}. Hence by \eqref{nsm3} with $\nu = \mu \eta$,  we have the estimate
\begin{align}
& \triangle_b (|d_b f|^2 + \mu \eta |f_0|^2) \geq \frac{1}{2} \left( |\nabla_b^f d_b f|^2 + \mu \eta |\nabla_b^f f_0|^2 \right) +\frac{1}{2} |\nabla_b^f d_b f|^2 \nonumber\\
&  \quad \qquad  + \left[ 2n- 5 \mu C_2' ( \frac{1}{R^2}+\frac{\sqrt{k}}{R}) \right] |f_0|^2 -  32 \left( k+ \frac{1}{\mu \eta} \right) |d_b f|^2 . \label{nsm5}
\end{align}
In order to deal with the first term of the right side, we need the following Schwarz inequalities:
\begin{align}
|d_b |d_b f|^2|^2 \leq & 4 \: |d_b f|^2 \: |\nabla_b^f d_b f|^2 \label{cs1} , \\
|d_b |f_0|^2|^2 \leq & 4 \: |f_0|^2 |\nabla_b^f f_0|^2 . \label{cs2}
\end{align}
If $|d_b f|(x) \neq 0$ and $|f_0|(x) \neq 0$, then at $x$, we have
\begin{align*}
&\frac{1}{2} \left( |\nabla_b^f d_b f|^2 + \mu \eta |\nabla_b^f f_0|^2 \right)  \\
& \qquad \geq  \  \frac{1}{8} \left( \frac{|d_b |d_b f|^2|^2}{|d_b f|^2} + \mu \eta \frac{|d_b |f_0|^2|^2}{|f_0|^2}+ \mu \frac{|d_b \eta|^2}{\eta} |f_0|^2 \right)  \ - \frac{1}{8} \mu \frac{|d_b \eta|^2}{\eta} |f_0|^2  \\
& \qquad \geq  \ \frac{1}{24} \frac{|d_b (|d_b f|^2 + \mu \eta |f_0|^2)|^2}{|d_b f|^2 + \mu \eta |f_0|^2} - \frac{\mu C_2'}{8} \frac{1}{R^2} |f_0|^2.
\end{align*}
Substituting this inequality to \eqref{nsm5}, we get \eqref{nsm6}. If $|d_b f|(x) =0$ (or $|f_0|(x)=0$), we can directly discard the nonnegative term $\frac{1}{2} |\nabla_b^f d_b f|^2$ (or $\frac{1}{2} \mu \eta |\nabla_b^f f_0|^2$) from \eqref{nsm5} and use the Schwarz inequality \eqref{cs1} (or \eqref{cs2}) to obtain \eqref{nsm6}.
\end{proof}

Let $x$ be a maximum point of $\eta F$ on $B_{2R}(x_0)$. If $x$ is not in the cut locus of $x_0$, then $\eta$ is smooth near $x$. If $x$ is in the cut locus of $x_0$, we may remedy $\eta$ by the following consideration. Since $(M, HM, J_b, \theta)$ is complete, there exists a length minimizing geodesic curve $\gamma: [0,1] \rightarrow M$ which joins $x_0$ and $x$. Let $\epsilon$ be a small positive number. Along $\gamma$, $x$ is before the cut point of $\gamma (\epsilon)$. This guarantees that the modified function $\tilde{r}(z) = d_c (z,\gamma(\epsilon)) + \epsilon$ is smooth in the neighborhood of $x$. Moreover, triangle inequality implies that:
$$ r \leq \tilde{r}, \quad and \quad r(x) = \tilde{r} (x) . $$
Set $\tilde{\eta} = \psi (\frac{\tilde{r}}{R})$. Then $\tilde{\eta}$ is smooth near $x$ and
$$ \eta \geq \tilde{\eta}, \quad and \quad \eta(x) = \tilde{\eta}(x).  $$
This means that $x$ is still a maximum point of $\tilde{\eta} F$. Hence, we may assume without loss of generality that $r$ is already smooth near $x$.

\begin{lem} \label{es1}
If x is a nonzero maximum point of $\eta F$ on $B_{2R} (x_0)$, then at $x$, we have the estimate
\begin{equation} \label{nsm7}
0 \geq \left[ 2 \eta F  - 34n ( k+ \frac{1}{\mu} ) \right] \frac{|d_b f|^2}{b^2_R-\rho^2 \circ f}+ \left[  2n - 31 \mu C_2' ( \frac{1}{R^2}+ \frac{\sqrt{k}}{R} ) \right] \frac{F}{\mu} .
\end{equation}
\end{lem}

\begin{proof}
It is obvious that $x$ is still a maximum point of $\ln(\eta F)$ on $B_{2R} (x_0)$.
Since $\triangle_b$ is a degenerate elliptic operator, the maximum principle implies that at $x$,
\begin{align}
0 & =  d_b \ln(\eta F)  =  \frac{d_b \eta}{\eta} + \frac{d_b (|d_b f|^2 + \mu \eta |f_0|^2)}{|d_b f|^2 + \mu \eta |f_0|^2} +\frac{d_b (\rho^2 \circ f)}{b^2_R-\rho^2 \circ f}, \label{nsm1}\\
0 & \geq  \triangle_b \ln(\eta F)  =  \frac{\triangle_b \eta}{\eta} - \frac{|d_b \eta|^2}{\eta^2} + \frac{\triangle_b (|d_b f|^2 + \mu \eta |f_0|^2)}{|d_b f|^2 + \mu \eta |f_0|^2}  \nonumber\\
 & \qquad -\frac{|d_b (|d_b f|^2 + \mu \eta |f_0|^2)|^2}{(|d_b f|^2 + \mu \eta |f_0|^2)^2}+ \frac{\triangle_b (\rho^2 \circ f)}{b^2_R-\rho^2 \circ f} +\frac{|d_b (\rho^2 \circ f)|^2}{(b^2_R- \rho^2 \circ f)^2}. \label{nsm2}
\end{align}
By Lemma \ref{l1}, \eqref{nsm2} becomes
\begin{align*}
0\geq &\   \frac{\triangle_b \eta}{\eta} - \frac{|d_b \eta|^2}{\eta^2} - \frac{23}{24} \frac{|d_b (|d_b f|^2 + \mu \eta |f_0|^2)|^2}{(|d_b f|^2 + \mu \eta |f_0|^2)^2}+\frac{|d_b (\rho^2 \circ f)|^2}{(b^2_R-\rho^2 \circ f)^2} \\
& \ + \frac{ [2n- 6 \mu C_2' ( \frac{1}{R^2}+\frac{\sqrt{k}}{R}) ] |f_0|^2 -  32( k+ \frac{1}{\mu \eta} ) |d_b f|^2}{|d_b f|^2 + \mu \eta |f_0|^2} + \frac{\triangle_b (\rho^2 \circ f)}{b^2_R-\rho^2 \circ f} .
\end{align*}
Substituting \eqref{nsm1} in above inequality and using Schwarz inequality: $(\alpha + \beta)^2 \leq 24 \alpha ^2 + \frac{24}{23} \beta^2$, we obtain
\begin{align*}
0 \geq &\ \frac{\triangle_b \eta}{\eta} - 24 \frac{|d_b \eta|^2}{\eta^2} - 32 \left(k + \frac{1}{\mu \eta} \right) \frac{|d_b f|^2 }{|d_b f|^2 + \mu \eta |f_0|^2}  \\
 &\quad + \left[2n- 6 \mu C_2' \left(\frac{1}{R^2}+\frac{\sqrt{k}}{R} \right)\right]\frac{|f_0|^2 }{|d_b f|^2 + \mu \eta |f_0|^2} +\frac{\triangle_b (\rho^2 \circ f)}{b^2_R-\rho^2 \circ f}.
\end{align*}
By the estimates \eqref{nd1} and \eqref{nseta}, we have
\begin{align}
0 \geq &\ - 25 \frac{C_2'}{\eta} \: \left(\frac{1}{R^2} +\frac{\sqrt{k}}{R} \right) - 32 \left(k + \frac{1}{\mu \eta} \right) \frac{|d_b f|^2 }{|d_b f|^2 + \mu \eta |f_0|^2} \nonumber \\
 &\quad + \left[2n - 6 \mu C_2' \left(\frac{1}{R^2}+\frac{\sqrt{k}}{R} \right)\right]\frac{|f_0|^2 }{|d_b f|^2 + \mu \eta |f_0|^2} +2 \frac{|d_b f|^2}{b^2_R- \rho^2 \circ f} . \nonumber
\end{align}
Hence multiplying both sides by $\eta F$, we conclude that
\begin{align}
0 \geq & -25 C_2' \: \left(\frac{1}{R^2}+\frac{\sqrt{k}}{R} \right)  F - 32 \left( \eta k + \frac{1}{\mu} \right) \frac{|d_b f|^2}{b^2_R -\rho^2 \circ f}   \nonumber \\
& \quad + \left[2n - 6 \mu C_2' (\frac{1}{R^2}+\frac{\sqrt{k}}{R}) \right] \frac{\eta |f_0|^2}{b^2_R-\rho^2 \circ f} + 2 \eta F \frac{|d_b f|^2}{b^2_R -\rho^2 \circ f}.  \label{nsm8}
\end{align}
Finally, we rewrite \eqref{nf} as
$$\frac{ \eta |f_0|^2}{b^2_R -\rho^2 \circ f} = \frac{1}{\mu} (F- \frac{|d_b f|^2}{b^2_R -\rho^2 \circ f}) $$
and substitute it into the previous inequality. This procedure yields
\begin{align*}
0 &\geq \  \left[  2n  - 31 \mu C_2' ( \frac{1}{R^2}+ \frac{\sqrt{k}}{R} ) \right] \frac{F}{\mu} \\
&\qquad + \left[ 2 \eta F - \frac{1}{\mu} \left(  2n   -  6\mu C_2' ( \frac{1}{R^2}+ \frac{\sqrt{k}}{R} ) \right) - 32 ( \eta k+ \frac{1}{\mu}) \right] \frac{|d_b f|^2}{b^2_R-\rho^2 \circ f}  \\
&\geq \ \left[  2n   - 31 \mu C_2' ( \frac{1}{R^2}+ \frac{\sqrt{k}}{R} ) \right] \frac{F}{\mu}+\left[ 2 \eta F - \frac{2 n}{\mu}  - 32 ( k+ \frac{1}{\mu} ) \right] \frac{|d_b f|^2}{b^2_R-\rho^2 \circ f}.
\end{align*}
The last inequality is due to $0 \leq \eta \leq 1$. Since $n \geq 1$, we get \eqref{nsm7}.
\end{proof}

Now we present our main results.
\begin{thm} \label{nsg}
Let $(M, HM, J_b, \theta)$ be a noncompact complete $(2n+1)$-Sasakian manifold with CR sub-Laplace comparison property relative to a fixed point $x_0$ and
\begin{equation*}
 Ric(X,X)  \geq -k |X|^2
\end{equation*}
for all $X \in T_{1,0} M$, and some $k\geq 0$. Suppose that $(N,h)$ is a simply connected Riemannian manifold with nonpositive sectional curvature. Assume that $f: M \rightarrow N$ is a pseudoharmonic map. Let $\rho $ be the Riemannian distance to $y_0 = f(x_0)$. For any $R>1$, set $b_R =2 \: sup \: \{ \rho \circ f(x) | x \in B_{2R} (x_0) \}$ and $a=\frac{R^2}{1 + \sqrt{k} R}$. Then, on $B_R (x_0)$
\begin{equation} \label{nsge}
|d_b f|^2 + a |f_0|^2 \leq C_3 \: b_R^2 \: \left( \frac{1}{a} + k \right)
\end{equation}
where the constant $C_3$ only depends on the dimension of $M$ and $C_1$.
\end{thm}

\begin{rmk}
Our auxiliary function \eqref{nf} for the maximum principle is slightly different from that one introduced in \cite{ckt}. In our case, we omit the variable $t$ in the auxiliary function. This seems to simplify the related estimates even for the pseudoharmonic function case.
\end{rmk}

\begin{proof}
Let $\mu = \frac{n}{31 C_2'} \frac{R^2}{1+ \sqrt{k}R}= \frac{n}{31 C_2'} a$. We consider the auxiliary function $F$ given by \eqref{nf}. Let $x$ be a maximum point of $\eta F$ on $B_{2R} (x_0)$. We assume $(\eta F) (x) \neq 0$ (Otherwise, the following estimate \eqref{n2} is trivial).
Since $2n- 31 \mu C_2' (\frac{1}{R^2} + \frac{\sqrt{k}}{R}) = n >0$,
the last term of the right side in \eqref{nsm7} is positve.
Hence Lemma \ref{es1} yields
\begin{equation} \label{n2}
\max_{z \in B_{2R} (x_0)} \ (\eta F)(z) \leq 17n \left( k+\frac{1}{\mu}\right).
\end{equation}
Since $\eta(z)=1$ for $z \in B_{R}(x_0)$, this inequality  asserts that on $B_{R}(x_0)$
\begin{equation*}
|d_b f|^2 + \mu |f_0|^2 \leq 17n (b_R^2- \rho^2 \circ f ) \left( k+\frac{1}{\mu}  \right) \leq 17n b_R^2 \left( k+\frac{1}{\mu}  \right).
\end{equation*}
Hence \eqref{nsge} can be obtained by choosing a proper constant $C_3$.
\end{proof}

The Reeb energy density is defined by the partial energy density $\frac{1}{2} |df(T)|^2$. From the sub-gradient estimate \eqref{nsge}, we can derive an estimate of Reeb energy density for pseudoharmonic maps and get some vanishing results.
\begin{cor} \label{plr}
Let $(M, HM, J_b, \theta)$ be a noncompact complete Sasakian manifold with CR sub-Laplace comparison property relative to a fixed point $x_0$ and
$$
 Ric(X,X)  \geq -k |X|^2
$$
for all $X \in T_{1,0} M$ and some $k \geq 0$. Suppose that $(N,h)$ is a simply connected Riemannian manifold with nonpositive sectional curvature. Assume that $f: M \rightarrow N$ is a pseudoharmonic map. Let $\rho $ be the Riemannian distance to $y_0 = f(x_0)$. For any $R>1$, set $b_R =2 \: sup \: \{ \rho \circ f(x) | x \in B_{2R} (x_0) \}$ and $a=\frac{R^2}{1 + \sqrt{k} R}$. Then, on $B_R (x_0)$
\begin{equation} \label{reeb}
|f_0|^2 \leq C_3 \: b_R^2 \left( \frac{2}{R^4} + \frac{3k}{R^2} + \frac{k\sqrt{k}}{R} \right) .
\end{equation}
In particular,
\begin{enumerate}[(i)]
\item if $Ric \geq 0$ (i.e. $k=0$) and the image of $f$ satisfies:
\begin{equation*}
\overline{\lim_{R \rightarrow \infty}} R^{-2} \: sup \: \{ \rho \circ f(x) | x \in B_{2R} (x_0)\} =0,
\end{equation*}
then $df(T)=0$.
\item if the pseudohermitian Ricci curvature of $M$ has strictly negative lower bound (i.e. $k>0$) and the image of f satisfies:
\begin{equation*}
\overline{\lim_{R \rightarrow \infty}} R^{-\frac{1}{2}} \: sup \: \{ \rho \circ f(x) | x \in B_{2R}(x_0) \} =0,
\end{equation*}
then $df(T)=0$.
\end{enumerate}
\end{cor}

The sub-gradient estimate \eqref{nsge} also gives Liouville theorem for pseudoharmonic maps.
\begin{thm} \label{cse}
Let $(M, HM, J_b, \theta)$ be a noncompact complete Sasakian manifold with nonnegative pseudohermitian Ricci curvature, and satisfy CR sub-Laplace comparison property relative to a fixed point $x_0 \in M$. Suppose that $(N,h)$ is a simply connected Riemannian manifold with nonpositive sectional curvature. Assume that $f: M \rightarrow N$ is a pseudoharmonic map. Let $\rho $ be the Riemannian distance to $y_0 = f(x_0)$. For any $R>1$, set $b_R =2 \: sup \: \{ \rho \circ f(x) | x \in B_{2R} (x_0) \}$. Then, on $B_R (x_0)$
\begin{equation*}
|d_b f|^2 + R^2 |f_0|^2 \leq C_3 \: \frac{b_R^2}{R^2}.
\end{equation*}
In particular, if the image of f satisfies
\begin{equation*}
\overline{\lim_{R \rightarrow \infty}} R^{-1} \: sup \: \{ \rho \circ f(x) | x \in B_{2R}(x_0) \} =0,
\end{equation*}
then f is a constant map.
\end{thm}

Since Heisenberg group $(\mathbb{H}^n, H \mathbb{H}^n, J_b, \theta)$ satisfies CR sub-Laplace comparison property, Theorem \ref{cse} can be applied to Heisenberg group.
\begin{cor} \label{hm1}
There is no bounded pseudoharmonic map from Heisenberg group $(\mathbb{H}^n, H \mathbb{H}^n, J_b, \theta)$ to a simply connected Riemannian manifold with nonpositive sectional curvature.
\end{cor}

\section{Appendix} \label{slh}
In this section, we will derive a Reeb energy density estimate for harmonic maps from Sasakian manifolds to Riemannian manifolds. We recall the definition of harmonic maps. Let $(M, HM, J_b, \theta)$ be a strictly pseudoconvex CR manifold, and let $\nabla^\theta$ be the Levi-Civita connection of $(M, g_\theta)$. Let $(N,h)$ be a Riemannian manifold, and $\hat{\nabla}$ its Levi-Civita connection. Suppose that $f: M \rightarrow N$ is a smooth map. Let $f^*TN$ be the pullback bundle and $\nabla^f$ the pullback connection. We can determine a connection $\nabla^{f, \theta}$ in $T^*M \otimes f^*TN$ by
$$
\nabla^{f, \theta}_X (\omega \otimes \xi )= \nabla^{ \theta}_X \omega \otimes \xi + \omega \otimes \nabla_X^f \: \xi$$
for any $X \in \Gamma(TM)$, $\omega \in \Gamma(T^*M)$ and $\xi \in \Gamma(f^*TN)$. So $f$ is harmonic if
$$
\tau^\theta (f;\theta,\hat{\nabla}) = trace_{g_\theta} (\nabla^{f,\theta} df) =0.
$$
With respect to the local orthonormal frame $\{ \theta, \theta^\alpha, \theta^{\ba}\}$ in $T^*M \otimes \mathbb{C}$ and $\{ \xi_i \}$ in $TN$, we have
\begin{align} \label{b1}
\tau^\theta (f;\theta,\hat{\nabla}) (f)  =  (f^i_{\alpha \ba} + f^i_{\ba \alpha} + f^i_{00} ) \xi_i.
\end{align}
Comparing with the equation \eqref{ph}, we obtain
\begin{equation} \label{phr}
\tau^\theta (f;\theta,\hat{\nabla}) (f) = \tau(f;\theta,\hat{\nabla})(f) + \nabla_T^f df(T).
\end{equation}

As above, we need a Bochner-type formula for harmonic maps and a special exhaustion function.
\begin{lem}
Let $f: M \rightarrow N$ be a smooth map. Then
\begin{align}
\frac{1}{2} \triangle |d f(T)|^2 =& \ | \nabla^f f_0 |^2 + \langle df(T), \nabla_T^f \; \tau^\theta (f;\theta,\hat{\nabla})\rangle  + 2 f^i_0 f^j_{\alpha} f^k_{\ba} f^l_0 \hat{R}_{j \ kl}^{\ i} \nonumber\\
&\  +2( f^i_0 f^i_\beta A_{\bb \ba , \alpha} + f^i_0 f^i_{\bb} A_{\beta \alpha , \ba} + f^i_0 f^i_{\bb \ba} A_{\beta \alpha} +f^i_0 f^i_{\beta \alpha} A_{\bb \ba} ) , \label{bf3}
\end{align}
where $\triangle$ is the Laplacian operator in $(M, g_\theta)$.
\end{lem}
\begin{proof}
On the one hand, we notice that
\begin{align} \label{b2}
\frac{1}{2} \triangle |d f(T)|^2 =& \frac{1}{2} \triangle_b |d f(T)|^2  + \frac{1}{2} (f^i_0 f^i_0)_{00}= \frac{1}{2} \triangle_b |d f(T)|^2 + f^i_{00} f^i_{00} + f^i_0 f^i_{000} .
\end{align}
On the other hand, by \eqref{b1}, we have
\begin{align*}
\langle df(T), \nabla_T^f \; \tau^\theta (f;\theta,\hat{\nabla})\rangle =& \langle df(T), \nabla_T^f \; \tau(f;\theta,\hat{\nabla})\rangle+ \langle df(T), \nabla_T^f \nabla_T^f df(T) \rangle \\
=& \langle df(T), \nabla_T^f \; \tau(f;\theta,\hat{\nabla})\rangle+ f^i_0 f^i_{000} .
\end{align*}
Hence substituting the above equation and \eqref{bf2} into \eqref{b2}, we get \eqref{bf3}.
\end{proof}

\begin{lem} \label{sb1}
Let $(M, HM, J_b, \theta)$ be a Sasakian manifold, and $(N,h)$ a Riemannian manifold with nonpositive sectional curvature. If $f: M \rightarrow N$ is a harmonic map, then
\begin{equation}
\frac{1}{2} \triangle |d f(T)|^2  \geq  | \nabla^f f_0 |^2 . \label{bte3}
\end{equation}
\end{lem}
The proof follows from \eqref{btep} and \eqref{bf3}.
\begin{dfn} \label{crch}
Let $(M, HM, J_b, \theta)$ be a Sasakian manifold with
$$Ric(X,X) \geq -k |X|^2$$
for any $X \in T_{1,0} M$, and some $k \geq 0$.
We say that $(M, HM, J_b, \theta)$ satisfies CR Laplace comparison property relative to a fixed point $x_0 \in M$, if there exists a positive constant $C_4$ such that the Carnot-Carath\'eodory distance $r$ to $x_0$ satisfies
\begin{eqnarray}
\triangle r & \leq & C_4 \: (\frac{1}{r}+ \sqrt{k}) \\
|d r|_{g_\theta} & \leq & C_4
\end{eqnarray}
on $M \setminus ( cut(x_0) \cup \{ x_0 \} )$ and where $r \geq 1$. 
\end{dfn}

On Heisenberg group $(\mathbb{H}^n, H \mathbb{H}^n, J_b, \theta)$, the square of the Carnot-Carath\'eodory distance function $r$ to the origin has the following expression
\begin{equation}
[r(z,t)]^2 = \frac{\phi^2}{(\sin \phi)^2} ||z||^2
\end{equation}
where $||z||^2 = \sum_{\alpha =1}^{n} |z^\alpha|^2$, $\phi$ is the unique solution of $\chi(\phi) ||z||^2 =|t|$ in the interval $[0, \pi)$ and $\chi(\phi)= \frac{\phi}{(\sin \phi)^2} - \cot \phi$. See \cite{ckt, ctw} for details.
\begin{prp} \label{hp2}
On Heisenberg group $(\mathbb{H}^n, H \mathbb{H}^n, J_b, \theta)$, there exists a positive constant $C_4'$ such that the Carnot-Carath\'eodory distance $r$ to the origin $o$ satisfies
\begin{eqnarray}
\triangle r & \leq & \frac{C_4'}{r} \label{h2}\\
|d r|_{g_\theta}^2 & \leq & C_4' \label{h3}
\end{eqnarray}
on $M \setminus ( cut(o) \cup \{ o \} )$ and where $r \geq 1$. Therefore, $(\mathbb{H}^n, H \mathbb{H}^n, J_b, \theta)$ satisfies CR Laplace comparison property relative to the origin.
\end{prp}
\begin{proof}
We first calculate $T r$ and $TT r$ on $M \setminus ( cut(o) \cup \{ o \} )$. When $t>0$, we take the partial derivative along $\frac{\partial}{\partial t}$ of $\chi(\phi) ||z||^2 =|t|$ and use the expression of $\chi$. The result is
\begin{equation*}
\frac{\partial \phi}{\partial t} = \frac{1}{2||z||^2} \: \frac{(\sin \phi)^3}{\sin \phi - \phi \cos \phi}.
\end{equation*}
Therefore,
\begin{align*}
T r^2 & =\ \frac{\partial r^2}{\partial t}  = \phi , \\
TT r^2 & =\ \frac{\partial^2 r^2}{\partial t^2} = \frac{1}{r^2} \frac{(\sin \phi)^5}{\phi^2 \: (\sin \phi - \phi \cos \phi)} .
\end{align*}
Since $TT r^2 = 2r \: TTr + 2 |Tr|^2$, there exists a constant $\tilde{C_4}$ such that
\begin{equation} \label{tr}
|Tr| \leq \frac{\tilde{C_4}}{r}, \quad |TTr| \leq \frac{\tilde{C_4}}{r^3} .
\end{equation}
When $t<0$, we can do the similar calculations and obtain the same inequality \eqref{tr}. When $t=0$, we can use the continuity property to get the same estimate \eqref{tr}, since $r$ is smooth on $M \setminus ( cut(o) \cup \{ o \} )$. Hence the inequalities \eqref{tr} always hold on $M \setminus ( cut(o) \cup \{ o \} )$. From Proposition \ref{hp1}, there exists a constant $\tilde{C_4'}$ such that
\begin{equation} \label{ccde}
\triangle_b r  \leq  \frac{\tilde{C_4'}}{r}
\end{equation}
on $M \setminus ( cut(o) \cup \{ o \} )$.
Let $C_4'=1+\tilde{C_4}+\tilde{C_4}^2+\tilde{C_4'}$. Then
\begin{eqnarray*}
\triangle r & =&  \triangle_b r + TTr \leq  \frac{C_4'}{r} \\
|d r|^2 & = & |d_b r |^2 + (Tr)^2 \leq C_4'
\end{eqnarray*}
on $M \setminus ( cut(o) \cup \{ o \} )$ and where $r \geq 1$.
\end{proof}

To derive the Reeb energy density estimate, we need an analogue estimate of \eqref{nd1}. Assume that $(N,h)$ is a simply connected Riemannian manifold with nonpositive sectional curvature. Let $\rho$ be the distance to a fixed point $y_0 \in N$.
If $f: M \rightarrow N$ is harmonic, the Hession comparison theorem implies
\begin{equation} \label{nd2}
\triangle (\rho^2 \circ f) \geq 2|d f|^2.
\end{equation}

\begin{thm} \label{csh}
Let $(M, HM, J_b, \theta)$ be a noncompact complete Sasakian manifold with CR Laplace comparison property relative to a fixed point $x_0$ and
$$Ric(X,X) \geq -k |X|^2$$
for any $X \in T_{1,0} M$, and some $k \geq 0$. Suppose that $(N,h)$ is a simply connected Riemannian manifold with nonpositive sectional curvature. Let $f: M \rightarrow N$ be a harmonic map. Let $\rho $ be the Riemannian distance to $y_0 = f(x_0)$. For any $R>1$, set $b_R =2 \: sup \: \{ \rho \circ f(x) | x \in B_{2R} (x_0)\}$. Then, on $B_R (x_0)$
\begin{equation} \label{he}
|df(T)|^2  \leq C_6 \: b_R^2 \: \left(\frac{1}{R^2}+ \frac{\sqrt{k}}{R}\right)
\end{equation}
where the constant $C_6$ depends only on $C_4$. Moreover,
\begin{enumerate}[(i)]
\item if $Ric \geq 0$ (i.e. $k=0$) and the image of f satisfies
\begin{equation*}
\overline{\lim_{R \rightarrow \infty}} R^{-1} \: sup \: \{ \rho \circ f(x) | x \in B_{2R}(x_0) \} =0,
\end{equation*}
then $df(T)=0$.
\item if  the pseudohermitian Ricci curvature of $M$ has strictly negative lower bound (i.e. $k>0$) and the image of f satisfies
\begin{equation*}
\overline{\lim_{R \rightarrow \infty}} R^{-\frac{1}{2}} \: sup \: \{ \rho \circ f(x) | x \in B_{2R} (x_0)\} =0,
\end{equation*}
then $df(T)=0$.
\end{enumerate}
\end{thm}

\begin{rmk}
In \cite{p}, R. Petit got a similar vanishing theorem for harmonic maps from compact Sasakian manifolds to Riemannian manifolds with nonpositive sectional curvature.
\end{rmk}

\begin{proof}
The choices of $\psi$ and $\eta$ are the same as in Section \ref{slp}. Since $(M, HM, J_b, \theta)$ satisfies CR Laplace comparison property, then $\eta$ satisfies
\begin{equation}
\begin{gathered}
\eta^{-1} |d \eta|^2  \leq  \frac{C_5}{R^2}  \\
\triangle \eta  = \frac{\psi ''}{R^2} |d r|^2 + \frac{\psi '}{R} \triangle r \geq - C_5 \: \left(\frac{1}{R^2}+ \frac{\sqrt{k}}{R}\right) \label{etah}
\end{gathered}
\end{equation}
\noindent on $M \setminus ( cut(x_0) \cup \{ x_0 \} )$. Here $C_5 $ depends only on $C_4$ and $C_2$.

Given $R >1$, we consider the function $G: M \rightarrow \mathbb{R}$, which is given by
$$ G(x) = \frac{|f_0|^2}{b_R^2-\rho^2 \circ f} (x) . $$
Let $x$ be a maximum point of $\eta G$ on $B_{2R} (x_0)$. If $x$ is in the cut locus of $x_0$, then we can modify $r$ as in Section \ref{slp}. Without loss of generality, assume that $r$ is smooth at $x$ and $ (\eta G)(x) \neq 0$. It is obvious that $x$ is still a maximum point of $\ln (\eta G)$ on $B_{2R} (x_0)$. Then the maximum principle asserts that at $x$,
\begin{align}
0 \  = d \ln (\eta G)= &\frac{d \eta}{\eta} + \frac{d |f_0|^2}{|f_0|^2} + \frac{d (\rho^2 \circ f)}{b^2_R- \rho^2 \circ f}, \label{hmp1}\\
0 \geq \triangle \ln (\eta G)= & \frac{\triangle \eta}{\eta} - \frac{|d \eta|^2}{\eta^2} + \frac{\triangle |f_0|^2}{|f_0|^2} - \frac{|d |f_0|^2|^2}{|f_0|^4} \nonumber \\
  &   \qquad+\frac{\triangle (\rho^2 \circ f)}{b^2_R-\rho^2 \circ f} +\frac{|d (\rho^2 \circ f)|^2}{(b^2_R- \rho^2 \circ f)^2}. \label{hmp2}
\end{align}
Applying \eqref{bte3} and the inequality $|d |f_0|^2|^2  \leq  4 \: |f_0|^2 |\nabla^f f_0|^2$ to \eqref{hmp2},
we have
\begin{equation*}
0   \geq  \frac{\triangle \eta}{\eta} - \frac{|d \eta|^2}{\eta^2} -\frac{1}{2} \frac{|d |f_0|^2|^2}{|f_0|^4} +\frac{|d (\rho^2 \circ f)|^2}{(b^2_R- \rho^2 \circ f)^2}+\frac{\triangle (\rho^2 \circ f)}{b^2_R-\rho^2 \circ f}.
\end{equation*}
With the aid of Schwarz inequality, we can use \eqref{hmp1} to estimate the third and fourth terms. The result is
\begin{equation*}
0 \geq \frac{\triangle \eta}{\eta} -2\: \frac{|d \eta|^2}{\eta^2} +\frac{\triangle (\rho^2 \circ f)}{b^2_R-\rho^2 \circ f}.
\end{equation*}
Therefore combining with \eqref{nd2} and \eqref{etah}, we conclude that at $x$,
\begin{equation*}
\frac{|d f|^2}{b^2_R-\rho^2 \circ f} \leq \frac{3 C_5}{ 2\eta } \left(\frac{1}{R^2}+ \frac{\sqrt{k}}{R}\right).
\end{equation*}
Hence by $|f_0|^2 \leq |d f|^2$, we can get an estimate of $\eta G$:
\begin{equation*}
\max_{z \in B_{2R}(x)} \frac{\eta |f_0|^2}{b^2_R-\rho^2 \circ f} \: (z) = (\eta G)(x) = \frac{\eta |f_0|^2}{b^2_R-\rho^2 \circ f} \: (x) \leq \frac{3 C_5}{2} \: \left(\frac{1}{R^2}+ \frac{\sqrt{k}}{R}\right).
\end{equation*}
This yields for any $z \in B_R (x_0)$,
\begin{equation*}
|f_0|^2 \: (z) \leq \frac{3 C_5}{2} \; (b^2_R-\rho^2 \circ f (z)) \left( \frac{1}{R^2}+ \frac{\sqrt{k}}{R} \right) \leq \frac{3 C_5}{2} \: b^2_R \: \left(\frac{1}{R^2}+ \frac{\sqrt{k}}{R}\right).
\end{equation*}
Let $C_6= \frac{3}{2} C_5$. The above inequality yields \eqref{he}.  The rest of this theorem follows from the estimate \eqref{he}.
\end{proof}

The relation \eqref{phr} shows that if $df(T) =0$, then harmonic map is equivalent to pseudoharmonic map. Therefore, Theorem \ref{cse} asserts the following Liouville theorem.

\begin{cor}
Let $(M, HM, J_b, \theta)$ be a noncompact complete Sasakian manifold with nonnegative pseudohermitian Ricci curvature, and satisfy both CR sub-Laplace comparison property and CR Laplace comparison property relative to a fixed point $x_0 \in M$. Suppose that $(N,h)$ is a simply connected Riemannian manifold with nonpositive sectional curvature. Assume that $f: M \rightarrow N$ is a harmonic map. Let $\rho$ be the Riemnnian distance to $y_0 = f(x_0)$. If the image of $f$ satisfies
\begin{equation*}
\overline{\lim_{R \rightarrow \infty}} R^{-1} \: sup \: \{ \rho \circ f(x) | x \in B_{2R}(x_0) \} =0,
\end{equation*}
then f is a constant map.
\end{cor}

Proposition \ref{hp1} and Proposition \ref{hp2} state that Heisenberg group satisfies both CR sub-Laplace comparison property and CR Laplace comparison property relative to the origin.

\begin{cor} \label{hm2}
There is no bounded harmonic map from Heisenberg group $(\mathbb{H}^n, H \mathbb{H}^n, J_b, \theta)$ to a simply connected Riemannian manifold with nonpositive sectional curvature.
\end{cor}

\begin{rmk}
If $n \geq 2$, then the Levi-Civita connection of  Heisenberg group $(\mathbb{H}^n, H \mathbb{H}^n, J_b, \theta)$ does not have nonnegative Ricci curvature. Thus Corollary \ref{hm2} can not be derived from the results in \cite{c}.
\end{rmk}

\section*{Acknowledgement}
The authors would like to express their thanks to Professor Yuxin Dong, Professor Qingchun Ji and Professor Jin-Hsin Cheng for constant encouragements and valuable discussions. Finally, we also wish to thank the referee for making the present form of the paper possible.


\begin{thebibliography}{}
\bibitem[1]{a} A. A. Agrachev, Exponential Mappings for Contact Sub-Riemannian Structures, J. Dynam. Control Systems, 2(1996), no.3, 321-358.
\bibitem[2]{c}  S. Y. Cheng, Liouville theorem for harmonic maps, Proc. Sympos. Pure Math.,vol.36, Amer. Math. Soc., Providence, RI, 1980, pp. 147-151.
\bibitem[3]{ckt} S. C. Chang, T. J. Kuo, and J. Tie, Yau's gradient estimate and Liouville Theorem for positive pseudoharmonic functions in a complete pseudohermitian manifold.
    http://www.math.sinica.edu.tw/www/file\_upload/conference
    /20126\_GPDE/presentation/06011-Chang2.pdf (Unpublished results).
\bibitem[4]{ctw} S. C. Chang, J. Tie and C. T. Wu, Subgradient Estimate and Liouville-type Theorems for the CR Heat Equation on Heisenberg groups, Asian J. Math., 14(2010), no.1, 041-072.
\bibitem[5]{dt} S. Dragomir and G. Tomassini, Differential Geometry and Analyis on CR manifolds, Progress in Mathematics, Vol. 246, Birkhauser 2006.
\bibitem[6]{g} A. Greenleaf, The first eigenvalue of a sub-Laplacian on a pseudo-Hermitian manifold, Comm. Partial Differential Equations 10 (1985), no. 2, 191-217.
\bibitem[7]{lee} J. M. Lee, Pseudo-Einstein Structure on CR Manifolds, Amer. J. Math. 110(1988), no.1, 157-178.
\bibitem[8]{p} R. Petit, Harmonic maps and strictly pseudoconvex CR manifolds, Comm. Anal. Geom., 10(2002), no.3, 575-610.
\bibitem[9]{s} R. Strichartz, Sub-Riemannian Geometry, J. Differential Geom. 24(1986), 221-263.
\bibitem[10]{y} S. T. Yau, Harmonic functions on complete Riemannian manifolds, Comm. Pure Appl. Math. 28 (1975), 201-228.
\bibitem[11]{w} S. M. Webster, Pseudohermitian structures on a real hypersurface, J. Differential Geom.13 (1978), 25-41.
\end{thebibliography}
\end{document}